\documentclass[11pt,leqno]{amsart}

\title[$L^1_{loc}$-convergence of Jacobians]{$L^{1}_{loc}$-convergence of Jacobians of Sobolev homeomorphisms via area formula}
\author[Z. Grochulska]{Zofia Grochulska}
\address{Zofia Grochulska, Department of Mathematics and Statistics, University of Jyväskylä, P.O. Box 35 (MaD), FI-40014 Finland\\ and  University of Warsaw, Banacha 2, 02-097 Warsaw, Poland} \email{zofia.z.grochulska@jyu.fi}
\thanks{Z.G.\ was supported by Polish National Science Centre grant no 2022/45/N/ST1/02977}


\usepackage{amsmath}
\usepackage{amsthm}
\usepackage[T1]{fontenc}
\usepackage{mathtools}
\usepackage{amsfonts, amssymb}
\usepackage{enumerate}
\usepackage{graphics}
\usepackage{float}
\usepackage{xcolor}
\usepackage{comment}
\usepackage[varbb]{newtxmath}  % for indicator function
\usepackage{hyperref}
            
\newcommand*{\R}{\mathbb{R}}
\newcommand*{\eps}{\varepsilon}
\newcommand*{\Lb}{\mathcal{L}}
\newcommand*{\Hs}{\mathcal{H}}

\newcommand*{\indi}{\mathbb{1}}

\newcommand{\mres}{\mathbin{\vrule height 1.6ex depth 0pt width
0.13ex\vrule height 0.13ex depth 0pt width 0.9ex}}  % restriction of measure 
\DeclareMathOperator*{\dist}{\mathrm{dist}}
\DeclareMathOperator*{\diam}{\mathrm{diam}}
\DeclareMathOperator*{\id}{\mathrm{id}}
\DeclareMathOperator*{\degg}{\mathrm{deg}}

\theoremstyle{definition}
\newtheorem{thm}{Theorem}
\newtheorem*{thm*}{Theorem}
\newtheorem{defn}[thm]{Definition}
\newtheorem{lemma}[thm]{Lemma}
\newtheorem*{lemma*}{Lemma}
\newtheorem{rmk}[thm]{Remark}
\newtheorem{example}[thm]{Example}
\newtheorem{cor}[thm]{Corollary}

\newtheorem{que}[thm]{Question}
\newtheorem{prop}[thm]{Proposition}
\newtheorem*{prop*}{Proposition}

\usepackage{titlesec}
\titlelabel{\thetitle.\quad}

\begin{document}
\begin{abstract}
 We prove that given a sequence of homeomorphisms $f_k: \Omega \to \R^n$ convergent in $W^{1,p}(\Omega, \R^n)$, $p \geq 1$ for $n =2$ and $p > n-1$ for $n \geq 3$, to a homeomorphism $f$ which maps sets of measure zero onto sets of measure zero, Jacobians $Jf_k$ converge to $Jf$ in $L^1_{loc}(\Omega)$. We prove it via Federer's area formula and investigation of when $|f_k(E)| \to |f(E)|$ as $k \to \infty$ for Borel subsets $E \Subset \Omega$.
\end{abstract}
\maketitle

\section{Introduction} \label{s1}

In this paper, we are considering a bounded domain $\Omega$ in $\R^n$ for $n \geq 2$, a class $\mathcal{F}$ of subsets of $\Omega$ and a sequence of mappings $f_k: \Omega \to \R^n$ converging in some sense to a mapping $f: \Omega \to \R^n$. We aim at finding conditions under which the following property holds
\begin{equation} \label{eq1} \tag{$\ast$}
    \lim_{k \to \infty} |f_k(E)| = |f(E)| \quad \text{for every set } E \in \mathcal{F},
\end{equation}
where $|\cdot|$ stands for Lebesgue measure in $\R^n$. We will focus on the situation when $f_k, f$ are homeomorphisms onto their respective images. 

Let us observe that this is a non-trivial property, not guaranteed by uniform convergence even when $\mathcal{F}$ is the family of all compact subsets of $\Omega$. In Section~\ref{s3}, we present the following
\begin{example} \label{T2a}
Let $Q=(0,1)^n$. There is a sequence $f_k:Q \to Q$ of Lipschitz homeomorphisms of $Q$ and a compact set $C \subset Q$ such that $f_k$ converge uniformly to identity on $Q$ but
\begin{equation} \label{eq3}
    \lim_{k \to \infty} |f_k(C)| = 0 < |C|.
\end{equation}
\end{example}

On the other hand, when $\mathcal{F}$ is the family of all compact subsets of $\Omega$, then the sets $f_k(E)$ converge to $f(E)$ in Hausdorff distance. Thus, the question whether \eqref{eq1} holds is a natural question about continuity of Lebesgue or Hausdorff measure with respect to Hausdorff convergence, see Lemma~\ref{T9}. It is standard to show that Lebesgue measure is \emph{upper} semicontinuous w.\,r.\,t.\ Hausdorff convergence. There is a number of criteria concerning the more difficult \emph{lower} semicontinuity; one of the most classical results in this direction is a theorem of Go\l{}\k{a}b stating that $\mathcal{H}^1$ is lower semicontinuous w.\,r.\,t.\ Hausdorff distance on a class of continua in $\R^n$, see \cite{golab, falconer}.

However, we are interested in the case when $f_k$ and $f$ are Sobolev or almost everywhere (a.\,e.) approximately differentiable homeomorphisms, see Section~\ref{s2} for a precise definition. This is rooted in the works of Ball in nonlinear elasticity \cite{Ball81, Ball82, Ball2001, Ball2010}. In fact, our main motivation for establishing property \eqref{eq1} for sequences of homeomorphisms stems from the link between \eqref{eq1} and $L^1_{loc}$-convergence of Jacobians. Let us recall that a mapping satisfies \emph{the Lusin condition (N)} if it maps sets of measure zero onto sets of measure zero. This property is often studied and assumed in nonlinear elasticity due to its physical meaning of not creating matter out of nowhere.

Our main result is
\begin{thm} \label{1T2}
Let $\Omega$ be a~bounded domain in $\R^n$ and $f_k, f: \Omega \to \R^n$ be homeomorphisms onto their respective images, $p \geq 1$ for $n =2$ and $p > n-1$ for $n \geq 3$. Assume that $f_k$ converge to $f$ in $W^{1,p}_{loc}(\Omega, \R^n)$ and that $f$ satisfies the Lusin condition (N). Then \eqref{eq1} holds for every Borel set $E \Subset \Omega$ and the Jacobians $Jf_k$ converge in $L^1_{loc}(\Omega)$ to the Jacobian $Jf$.
\end{thm}
The assumption that $f$ satisfies the Lusin condition (N) is necessary. Ponomarev in \cite{Ponomarev} showed that for $p < n$, there is a homeomorphism $f \in W^{1,p}((0,1)^n, \R^n)$ which maps a set of zero measure $C$ onto a set of positive measure. It can be constructed as a $W^{1,p}$-limit of a sequence of Lipschitz homeomorphisms $f_k$ (see \cite{KauhanenKoskelaMaly} and \cite[Chapter 4]{Distortion}) for a modern treatment of this construction). Therefore, $|f_k(C)|=0$ whereas $|f(C)| > 0$, which contradicts \eqref{eq1}.

Also, it is important that $E$ is a Borel set. Since $f_k$ do not need to satisfy the Lusin condition (N), $f_k(E)$ may not be measurable if $E$ is assumed to be only measurable. In fact, any function without the Lusin condition (N) (e.\,g., the Cantor function) maps some measurable set of measure zero onto a non-measurable set. 

One of the most challenging questions about Sobolev homeomorphisms, which stems from nonlinear elasticity, is the Ball--Evans question \cite{Ball2001, Ball2010, evans}: given a homeomorphism $f \in W^{1,p}(\Omega, \R^n)$, does there exist a sequence of diffeomorphisms converging to $f$ in $W^{1,p}(\Omega, \R^n)$? It remains unsolved in dimension $n=3$, the answer is \emph{no} in dimensions $n \geq 4$ (\cite{HenclVejnar2016}) and \emph{yes} in dimension $n=2$. This last result was shown for $p \geq 1$ by Iwaniec, Kovalev, Onninen in \cite{Iwaniec2010} and for $p=1$ by Hencl and Pratelli in \cite{HenclPratelli2018}. They used an earlier result by Mora--Corral and Pratelli \cite{MoraCorral2014}. These important results together with Theorem~\ref{1T2} imply the following

\begin{cor}[\cite{HenclPratelli2018}, \cite{MoraCorral2014}, Theorem~\ref{1T2}]
Suppose that $\Omega$ is a bounded domain in $\R^2$ and that $f \in W^{1,p}(\Omega, \R^2)$, $p \geq 1$, is a homeomorphism which satisfies the Lusin condition (N). Then, there exists a sequence of diffeomorphisms $f_k$ converging to $f$ in $W^{1,p}(\Omega, \R^2)$ and $Jf_k$ converge to $Jf$ in $L^{1}_{loc}(\Omega)$.
\end{cor}

The proof of Theorem~\ref{1T2} consists of two observations, which we state below as Theorems \ref{T1} and \ref{1T1}. We state them in a more general setting than needed for the proof of Theorem~\ref{1T2} as we believe they are of independent interest. We say that a mapping $g: \Omega \to \R^n$ is injective a.\,e.\ if there is a set of measure zero $N \subset \Omega$ such that $f|_{\Omega \setminus N}$ is injective. This is sometimes called injectivity a.\,e.\ in the domain. Injectivity a.\,e.\ is a notion often used in nonlinear elasticity and related research, see \cite{Ball81, MullerSpector, Barchiesi, Bouchala}.

The first observation concerns the already mentioned link between \eqref{eq1} and $L^{1}_{loc}$-convergence of Jacobians. We state it below, denoting the approximate derivative of $f$ as $D_{\rm a}f$ and the approximate Jacobian as $J_{\rm a }f$.

\begin{thm} \label{T1}
Let $\Omega$ be a bounded domain in $\R^n$ and $f_k: \Omega \to \R^n$ be continuous, injective a.\,e.\ and approximately differentiable a.\,e. Suppose that $f: \Omega \to \R^n$ satisfies the Lusin condition (N) and is approximately differentiable a.\,e.\ and that $J_{\rm a}f$ is locally integrable. Assume that $D_{\rm a} f_k$ converge to $D_{\rm a}f$ in $\mathcal{L}^n$-measure and that for every Borel set $E \Subset \Omega$, \eqref{eq1} holds. Then $J_{\rm a} f_k$ converge to $J_{\rm a} f$ in $L^1_{loc}(\Omega)$.
\end{thm}

The proof of Theorem~\ref{T1} relies mainly on the area formula originally proved by Federer in \cite{federer2}. In particular, he proved that the area formula (and hence the change of variables formula) holds for a.\,e.\ approximately differentiable homeomorphisms which satisfy the Lusin condition (N). This formulation of this general and deep theorem seems to be relatively unknown and by describing it in detail in Section~\ref{s2}, we aim at advertising its usefulness. One of its immediate corollaries is the fact that for an a.\,e.\ approximately differentiable homeomorphism $f$, $J_{\rm a} f$ is locally integrable, see Remark~\ref{T27}.

Let us note that the assumptions of Theorem~\ref{T1} are not sufficient for \eqref{eq1} to hold, even if $f_k, f$ are homeomorphisms which satisfy the Lusin condition (N), as seen in
\begin{example} \label{1T3}
Let $Q = (0,1)^n$. There are a.\,e.\ approximately differentiable homeomorphisms $f_k, f:Q \to \R^n$, which satisfy the Lusin condition (N) such that
\begin{enumerate}[(i)]
    \item $D_{\rm a}f_k$ converges to $D_{\rm a}f$ in $\mathcal{L}^n$-measure;
    \item there is a compact set $C \subset \Omega$ with $\lim_{k \to \infty} |f_k(C)| > |f(C)|$;
    \item $\{J_{\rm a } f_k\}_k$ are not uniformly integrable on $Q$.
\end{enumerate}
\end{example}

The second ingredient of the proof of Theorem~\ref{1T2} is Theorem~\ref{1T1} below, which gives a sufficient condition for \eqref{eq1} to hold.
\begin{thm} \label{1T1}
Let $\Omega$ be a bounded domain in $\R^n$ and $f_k: \Omega \to \R^n$ be continuous, injective a.\,e.\ and approximately differentiable a.\,e. Assume that $f_k$ converge locally uniformly to $f: \Omega \to \R^n$, that $D_{\rm a}f_k$ converge in $\mathcal{L}^n$-measure to $D_{\rm a}f$ and that $f$ satisfies the Lusin condition (N). Then \eqref{eq1} holds for any Borel set $E \Subset \Omega$.
\end{thm}

In fact, to prove Theorem~\ref{1T2}, we show the folklore result that given a sequence of Sobolev homeomorphisms $f_k$ convergent in $W^{1,p}_{loc}$ for $p$ as in Theorem~\ref{1T2}, $f_k$ converges locally uniformly, see Proposition~\ref{T16}.

This follows from the assumption $p > n-1$ and the use of Morrey's inequality on $(n-1)$-dimensional hyperplanes. This also gives rise to the following
\begin{que}
Let $\Omega$ be a bounded domain in $\R^n$, $n \geq 3$, and $f_k, f: \Omega \to \R^n$ be homeomorphisms onto their respective images, $p \leq n-1$. Assume that $f_k$ converge to $f$ in $W^{1,p}_{loc}(\Omega, \R^n)$ and that $f$ satisfies the Lusin condition (N). Does \eqref{eq1} hold for every Borel set $E \Subset \Omega$?
\end{que}
\noindent Should the answer be \emph{yes}, Theorem~\ref{T1} would imply $L^{1}_{loc}$-convergence of Jacobians. It is natural to suppose that the answer is negative if $p < n-1$. The case $p = n-1$ seems more challenging.

From the point of view of calculus of variations or nonlinear elasticity, one would also be interested in sequences of homeomorphisms or a.\,e.\ injective mappings whose limit is not a homeomorphism. Obviously, without this assumption, checking \eqref{eq1} or $L^1_{loc}$-convergence of Jacobians is much more complex. In \cite{Sbornik} (see also \cite{alexandrova99} and \cite[Chapter 3.6]{BogachevBook}), the authors discussed property \eqref{eq1} from a more probabilistic angle, not focusing on homeomorphisms. They were interested in convergence in variation of pushforward measures (or images of measures) and their results found application in probability, see for example \cite{breton, kulik}.

Using essentially the same argument as the authors of \cite{Sbornik} (compare with Corollary 2.5, Corollary 2.6), we prove
\begin{thm} \label{T6}
Let $\Omega$ be a~bounded domain in $\R^n$ and $f_k: \Omega \to \R^n$ be continuous mappings converging locally uniformly to a~mapping $f: \Omega \to \R^n$. Assume that $f_k, f$ satisfy the Lusin condition (N) and are a.\,e.\ approximately differentiable. Moreover, assume that $\{J_{\rm a} f_k \}_k$ is uniformly integrable on every compact subset of $\Omega$. Suppose that for any subdomain $U \Subset \Omega$ with $|\partial U| = 0$,
\begin{equation} \label{eq4}
    \degg \, (f, U, y) \neq 0 \text{ for a.\,e.\ } y \in f(U).
\end{equation}
Then for any measurable subset $E \Subset \Omega$, \eqref{eq1} holds.
\end{thm}
\noindent Here $\degg(f, U, y)$ stands for the topological degree, described in detail in Section \ref{s6}. Any homeomorphism satisfies \eqref{eq4}; in Section~\ref{s6} we state a weaker condition which also guarantees \eqref{eq4}. Even though Theorem~\ref{T6} does not talk about $L^{1}_{loc}$-convergence of Jacobians, we decided to include it to show what can be said about non-injective mappings and to point out the extent of use of area formula. Our proof clarifies and simplifies the reasoning from \cite{Sbornik}, yet it follows exactly the same line of thought as the proof of \cite[Corollary 2.5]{Sbornik}.

The paper is structured as follows. Section~\ref{s2} contains preliminaries, Section~\ref{s3} describes Examples \ref{T2a} and \ref{1T3}. In Section~\ref{s4}, we prove Theorems \ref{T1} and \ref{1T1} and in Section~\ref{s5} we prove the main Theorem~\ref{1T2}. In the last Section~\ref{s6}, we prove Theorem~\ref{T6}.

\subsection*{Acknowledgments} I would like to thank Stanislav Hencl, Pawe\l{} Goldstein and Zheng Zhu for encouragement and discussions.

\section{Preliminaries} \label{s2}
\subsection*{Notation.} Throughout this section and the rest of the paper, $| E |$ stands for Lebesgue measure of a~set $E$ and \emph{measure} and \emph{measurable} always refers to the Lebesgue measure. A \emph{homeomorphism} means a~\emph{homeomorphism onto the image} and, if not specified otherwise, $\Omega$ is a~bounded domain in $\R^n$. For $x \in \R^n$, we write $|x| = (\sum_{i=1}^n |x_i|^2)^{1/2}$ and, given a set $E \subset \R^n$, $\dist(x, E) = \inf \{ |x-y|: \, y \in E\}$. Given a compact set $C$ and a positive number $\delta$, we denote by $C_\delta$ the $\delta$-neighborhood of set $C$, i.\,e., $C_\delta := \{x: \, \dist(x, C) < \delta\}$. With $\lVert \cdot \rVert_{1,p, E}$ we denote the Sobolev $W^{1,p}(E, \R^n)$-norm and with $\indi_E$ the indicator function of the set $E$ (i.\,e.\ $\indi_E = 1$ on $E$ and $\indi_E = 0$ outisde $E$).

\subsection*{Measurability.}
If $f: \Omega \to \R^n$ is continuous and injective, it is standard to show that $f$ maps Borel sets onto Borel sets. Also, if $f$ is continuous and it satisfies the Lusin condition (N), it is standard to show that $f$ maps measurable sets onto measurable sets. Nonetheless, we shall need the following highly nontrvial
\begin{thm} \label{T17}
Let $f: \Omega \to \R^n$ be continuous. Then for any Borel set $E \subset \Omega$, $f(E)$ is measurable.
\end{thm}
\noindent The statement and proof can be found in \cite[Theorem 2.2.13]{Federer}, the theorem itself goes back to Suslin.

\subsection*{Approximate differentiability.}
\begin{defn}
Let $\Omega$ be an open set in $\R^n$. We say that a~measurable mapping $f: \Omega \to \R^n$ is approximately differentiable at a~point $x_o \in \Omega$ if there is a~measurable set $E_o$ of density $1$ at $x_o$ and a~linear mapping $L_o: \Omega \to \R^{n \times n}$ such that
$$
\lim_{y \to x_o, \, y \in E_o} \frac{|f(x_o) - f(y) - L_o(x_o - y)|}{ |x_o - y  |} = 0.
$$
We then denote with $D_{\rm a} f(x_o) := L_o$, the approximate derivative of $f$ at $x_o$ and with $J_{\rm a} f(x_o) := \det D_{\rm a} f(x_o)$, the approximate Jacobian of $f$ at $x_o$.
\end{defn}
\begin{rmk}
Whitney in \cite{whitney1951} characterized mappings which are a.\,e.\ approximately differentiable as precisely these mappings which coincide with $C^1$ mappings on sets of arbitrarily large measure. That is, $f$ is a.\,e.\ approximately differentiable if and only if for any $\eps >0$, there is a $C^1$ mapping $f_\eps$ s.\,t.\ $|\{f \neq f_\eps \}| < \eps$.

If a measurable function $\tilde{f}$ coincides with an a.\,e.\ approximately differentiable function $f$ on $\Omega$, then $f$ is also a.\,e.\ approximately differentiable and $D_{\rm a} f(x) = D_{\rm a} \tilde{f}(x)$ for almost every $x \in \Omega$. For a~short overview, see \cite[Section 6]{GGH1}.

Let us recall that if $f \in W^{1,p} (\Omega, \R^n)$ for any $p \geq 1$, then $f$ is approximately differentiable a.\,e. and the weak derivative $Df$ coincides with the approximate derivative $D_{\rm a} f$ a.\,e.\ on $\Omega$, see \cite[Theorem 4, Chapter 6]{EvansGariepy}.
\end{rmk}

\subsection*{Area formula.} The main tool of this paper is the area formula proved by Federer in~\cite{federer2}, see also \cite[Theorem 3.2.3]{Federer}. The exact formulation quoted below comes from~\cite{hajlasz}. The function $y \mapsto N(f, \Omega, y)$ equals the cardinality of the set $f^{-1}(y)$ provided this set is finite and $\infty$ if this set is infinite\footnote{The function is often called the multiplicity function or the Banach indicatrix}.
\begin{thm}[Federer] \label{1T7}
Let $\Omega$ be an open set in $\R^n$ and $f: \Omega \to \R^n$ be a.\,e.\ approximately differentiable. If $f$ satisfies the Lusin condition (N), then for any measurable function $\varphi: \R^n \to \R$, we have
\begin{equation} \label{eq25}
\int_{\Omega} (\varphi \circ f)(x) |J_{\rm a} f(x)| \, dx = \int_{f(\Omega)} \varphi(y) N(f, \Omega, y) \, dy. 
\end{equation}
If $f$ does not satisfy the Lusin condition (N), $f$ can be redefined on a Borel set of measure zero so that after the redefinition, $f$ satisfies the Lusin condition (N) and \eqref{eq25} holds.
\end{thm}

\begin{cor} \label{T26}
    Let $\Omega$ be an open set in $\R^n$ and let $f:\Omega \to \R^n$ be continuous and a.\,e.\ approximately differentiable. 
     \begin{enumerate}[(i)]
        \item If $f$ satisfies the Lusin condition (N), then for any measurable set $A \subset \Omega$, we have $\int_{A} |J_{\rm a} f(x)| \, dx \geq |f(A)|$.
    \end{enumerate}
        Additionally, assume that $f$ is injective a.\,e. Then
    \begin{enumerate}[(i)]
        \setcounter{enumi}{1}
        \item for any Borel set $E \subset \Omega$, $\int_{E} |J_{\rm a} f(x)| \, dx \leq |f(E)|$;
        \item for any measurable set $A$, $A \subset \Omega$, on which $f$ satisfies the Lusin condition (N), $\int_A |J_{\rm a}f(x)| \, dx = |f(A)|$.
    \end{enumerate}
\end{cor}
\begin{proof}
Claim (i) follows from applying Theorem~\ref{1T7} to the measurable function $\indi_{f(A)}$ and observing that for any $y \in f(A)$, $N(f, \Omega, y) \geq 1$. 

Let us start proving (ii) by noting that by Theorem~\ref{T17}, $f(E)$ is measurable. Let $\tilde{f}$ be the redefinition of $f$ on the Borel set $Z$ of zero measure so that $\tilde{f}$ satisfies the Lusin condition (N) and $f = \tilde{f}$ on $\Omega \setminus Z$. This implies that $J_{\rm a} f= J_{\rm a} \tilde{f}$ a.\,e.\ on $\Omega \setminus Z$. Without loss of generality, we can assume that $f$ is injective on $\Omega \setminus Z$. Hence, $N(\tilde{f}, \Omega, y) =1$ a.\,e.\ on $\tilde{f}(E\setminus Z)$. Indeed, for any $y \in \tilde{f}(E \setminus Z) \setminus \tilde{f}(Z)$, the preimage $\tilde{f}^{-1}(y)$ is a singleton, because $\tilde{f}=f$ on $\Omega \setminus Z$ and $f$ is injective on that set. Now, we apply Theorem~\ref{1T7} to the measurable function $\indi_{\tilde{f}(E\setminus Z)}$ to get
\begin{equation*}
    \int_{E} |J_{\rm a} f(x)| \, dx \leq \int_{\tilde{f}(E\setminus Z)} N(\tilde{f}, \Omega, y) \, dy = \int_{\tilde{f}(E \setminus Z)} 1 \, dy = |f(E \setminus Z)| \leq |f(E)|.
\end{equation*}

Claim (iii) follows from (i) and (ii). It is sufficient to assume that $A$ is measurable to know that $f(A)$ is measurable because $f$ satisfies the Lusin condition (N).
\end{proof}

\begin{rmk} \label{T27}
Let $f: \Omega \to \R^n$ be continuous, a.\,e.\ approximately differentiable and injective a.\,e. It follows from Corollary~\ref{T26} (ii) that $J_{\rm a}f$ is locally integrable.
\end{rmk}

\section{Examples \ref{T2a} and \ref{1T3}} \label{s3}

\begin{proof}[Example \ref{T2a}]
As mentioned in the Introduction, this example relies heavily on Ponomarev's example from~\cite{Ponomarev}, which was then described using \emph{frame-to-frame} mappings in~\cite{KauhanenKoskelaMaly} and then in~\cite[Chapter 4]{Distortion}. We refer the reader there for the precise description of the construction. Recall that $Q = (0,1)^n$.

By \cite[Theorem 4.10]{Distortion}, there is a homeomorphism $h: Q \to Q$ and two Cantor-type sets $C_p \subset Q$ and $C_z \subset Q$ such that
\begin{itemize}
    \item $|C_p| > 0$ and $|C_z| = 0$;
    \item $h$ maps $C_z$ onto $C_p$ (thus it does not satisfy the Lusin condition (N))
    \item $h$ is a uniform limit of Lipschitz homeomorphisms $h_k: Q \to Q$.
\end{itemize}
Homeomorphism $h$ is essentially the one discovered by Ponomarev in~\cite{Ponomarev}. It is possible to construct such $h$ so that $h \in W^{1,p}(Q,Q)$ for $p <n$, nonetheless we will not use its Sobolev regularity. For each $k$, $h_{k+1}$ is a~modification of $h_k$, so for large $k$, $h$ coincides with $h_k$ on a~set of large measure.

Set $g = h^{-1}$ so that $g$ maps the set $C_p$ onto the set $C_z$. By \cite[Theorem 4.15]{Distortion}, $g$ is Lipschitz and thus it satisfies the Lusin condition (N). It can be constructed in the same way as $h$---it is a uniform limit of $h_k^{-1}$ and $g$ coincides with $h_k^{-1}$ on a set of large measure.

We now set $f_k := h_k \circ g$. Mappings $f_k$ are clearly Lipschitz homeomorphisms which converge uniformly to $f=\id$. Observe that, as $k$ increases, $f_k$ coincides with $\id$ on larger and larger sets. Since $g(C_p) = C_z$ and $h_k$ are Lipschitz, for every $k$ we have
$$
|f_k(C_p)| = 0 \text{ whereas } |f(C_p)| = |C_p| > 0,
$$
which shows that \eqref{eq3} holds and that \eqref{eq1} does not, even for compact subsets of $Q$.

Note that it follows from the area formula (Corollary~\ref{T26}~(iii)) that $Jf_k = 0$ a.\,e.\ on a~set of positive measure $C_p$. This implies that $Jf_k$ does not converge in measure to $Jf=1$ and, as a~result, that $Df_k$ do not converge in measure to $Df = \mathrm{Id}$, see Lemma~\ref{T13} for details.
\end{proof}

\begin{proof}[Example \ref{1T3}]
Let $P_k$ be the cube centered at $(1/2, \ldots, 1/2)$ of side length $1/k$, $k=1, 2, \ldots$ with sides parallel to coordinate axes. Set $g(x) := \tfrac{1}{2} \indi_Q (x)$ and
\begin{equation*}
g_k(x) := \tfrac{k^n}{2} \indi_{P_k} (x) + \tfrac{1}{2(1 - 1/k^n)} \indi_{Q \setminus P_k}(x).
\end{equation*}
Note that $g_k$ converge to $g$ in measure and that $\int_Q |g_k(x)| \, dx = 1$ for all $k$. Let
\begin{equation*}
   T_k (x) :=
\left[
\begin{array}{ccccc}
g_k(x)       &      0       &   \ldots   &   0      &     0      \\
0       &      1       &   \ldots   &   0      &     0      \\
\vdots  &   \vdots     &   \ddots   &  \vdots  &    \vdots  \\
0       &      0       &   \ldots   &   1      &     0      \\
0       &      0       &   \ldots   &   0      &    1      \\
\end{array}
\right]\ \
 \text{ and } T(x):=
 \left[
\begin{array}{ccccc}
g(x)       &      0       &   \ldots   &   0      &     0      \\
0       &      1       &   \ldots   &   0      &     0      \\
\vdots  &   \vdots     &   \ddots   &  \vdots  &    \vdots  \\
0       &      0       &   \ldots   &   1      &     0      \\
0       &      0       &   \ldots   &   0      &    1      \\
\end{array}
\right].
\end{equation*}
Then $T_k$ converge in measure to $T$.

By \cite[Theorem 1.4]{GGH1}, we find a.\,e.\ approximately differentiable homeomorphisms $f_k, f: Q \to \R^n$, which satisfy the Lusin condition (N) and such that $D_{\rm a} f_k = T_k$, $D_{\rm a}f = T$ a.\,e on $Q$. Therefore, $D_{\rm a} f_k$ converge in measure to $D_{\rm a} f$. Clearly, $J_{\rm a} f_k = g_k$. It is a standard exercise to check that $\{g_k\}_k$ are not uniformly integrable. This shows that the properties (i) and (iii) hold.

We now check property (ii). Let $C$ be the cube centered at $(1/2, \ldots, 1/2)$ of side length $1/2$ with sides parallel to coordinate axes. Then, since $P_k \subseteq C$ for all $k \geq 2$,
\begin{align*}
\lim_{k \to \infty} |f_k(C)| &= \lim_{k \to \infty} \int_{C} g_k(x) \, dx = \lim_{k \to \infty} \tfrac{1}{2} + \tfrac{1}{2(1-1/k^n)}(2^{-n} + k^{-n}) = \tfrac{1}{2} + \tfrac{1}{2^{n+1}} \\
&> |f(C)| = \tfrac{1}{2^{n+1}}.
\end{align*}
\end{proof}

\section{Proofs of Theorems \ref{T1} and \ref{1T1}} \label{s4}

\begin{lemma} \label{T13}
Let $\mu$ be a~finite Radon measure on $\R^n$ and $h_k: \R^n \to \R^n$ a~sequence of $\mu$-measurable mappings converging in measure $\mu$ to $h: \R^n \to \R^n$. Then, for any continuous function $\phi: \R^n \to \R$, functions $\phi \circ h_k$ converge to $\phi \circ h$ in measure $\mu$.
\end{lemma}
\begin{proof}
Fix $\eps > 0$. Given any $ \eta > 0$, Lusin's theorem yields a~compact set $K \subset \R^n$ with $\mu(\R^n \setminus K) < \eta/2$ on which $h$ is bounded by some positive number $M$. Function $\phi$ is uniformly continuous on the ball $\overline{B}(0, M+1)$, so we can choose $\delta \in (0,1)$ s.\,t.\ if $|y_1 - y_2| \leq \delta$, then $|\phi(y_1) - \phi(y_2)| < \eps$. Then we find $k_o$ such that the complement of the set $A_\delta:=\{x: \, |h_k(x) - h(x)| \leq \delta\}$ has measure smaller than $\eta/2$.

If $x \in K \cap A_\delta$, then $h_k(x), \, h(x) \in B(0, M+1)$ and $|\phi(h_k(x)) - \phi(h(x))| < \eps$. Consequently,
$$
\
\mu(\{x: |\phi(h_k(x)) - \phi(h(x))| > \eps\}) \leq \mu(\R^n \setminus K) +  \mu(\R^n \setminus A_\delta) < \eta.
$$
The claim follows from the arbitrariness of $\eps$ and $\eta$.
\end{proof}

We are now in the position to prove Theorem~\ref{T1}, which links property~\eqref{eq1} with $L^1$-convergence of Jacobians.

\begin{proof}[Proof of Theorem \ref{T1}]
Fix any compact set $K \subset \Omega$. Lemma~\ref{T13} for the finite measure $\Lb^n \mres \, \Omega$ implies that $\indi_K J_{\rm a} f_k$ converge to $\indi_K J_{\rm a} f$ in measure. In view of Vitali's convergence theorem, it suffices to show that the sequence $\{\indi_K J_{\rm a} f_k\}_k$ is uniformly integrable.

Fix $\eps >0$. Because $f$ satisfies the Lusin condition (N) and $\indi_K J_{\rm a}f$ is integrable, there is $\eta >0$ such that for every Borel set $E$ with $|E| < \eta$, $|f(E \cap K)| < \eps$. This follows from the area formula from Corollary~\ref{T26}~(i). By Corollary~\ref{T26}~(ii) and~\eqref{eq1}, there is $k_o$ such that for any Borel set $E$ with $|E| <\eta$,
$$
\sup_{k \geq k_o} \int_E |\indi_K J_{\rm a} f_k(x)| \, dx \leq \sup_{k \geq k_o} |f_k(E\cap K)| \leq 2 |f(E \cap K)| < 2\eps.
$$
This shows uniform integrability of the family $\{\indi_K J_{\rm a} f_k\}_k$ for $k \geq k_o$. Therefore, $\{\indi_K J_{\rm a} f_k\}_k$ for $k = 1, \ldots, k_o$ is uniformly integrable as well, since any finite family of integrable functions is uniformly integrable.
\end{proof}

Next, we prove Theorem~\ref{1T1}. We start with a lemma which essentially says that Lebesgue measure is upper semicontinuous w.\,r.\,t.\ Hausdorff convergence, as mentioned in the Introduction. Indeed, \eqref{eq51} proves that $f_k(K)$ converge to $f(K)$ in Hausdorff distance. Recall that given a compact set $C$ and a positive number $\delta$, we denote $C_\delta := \{x: \, \dist(x, C) < \delta\}$. The Hausdorff distance $d$ is defined for any nonempty compact sets $K, C$ as
$$
d(K, C) := \inf \{ \delta \geq 0: \, K \subset C_\delta \text{ and } C \subset K_\delta \}.
$$

\begin{lemma} \label{T9}
Let $f_k: \Omega \to \R^n$ be a~sequence of continuous mappings converging locally uniformly to a~mapping $f$. Then, for any compact set $K \subset \Omega$,
\begin{equation} \label{eq6}
\limsup_{k \to \infty} |f_k(K)| \leq |f(K)|.
\end{equation}
\end{lemma}
\begin{proof}
Clearly, $f$ is continuous and $f(K)$, $f_k(K)$ are compact. For any $\eps >0$, by continuity of measure, there is $\delta > 0$ such that $|(f(K))_\delta| \leq |f(K)| + \eps$. There is also $k_o$ such that for all $k \geq k_o$, for all $x \in K$, $|f_k(x) - f(x)| < \delta$. This implies that
\begin{equation} \label{eq51}
f_k(K) \subset (f(K))_\delta \text{ and } f_k(K) \subset (f(K))_\delta \text{ for all } k \geq k_o.
\end{equation}
Moreover,
$$
|f_k(K)| \leq |(f(K))_\delta| \leq |f(K)| + \eps.
$$
Since $\eps >0$ is arbitrary, \eqref{eq6} holds.
\end{proof}

\begin{proof}[Proof of Theorem~\ref{1T1}]
Let $k_j$ be the subsequence s.\,t.\ $\liminf_{k \to \infty} |f_k(E)| = \lim_{j \to \infty} |f_{k_j}(E)|$. By Lemma~\ref{T13}, we know that $J_{\rm a} f_k$ converge in measure to $J_{\rm a} f$. Therefore, we can choose a subsequence $k_{j_{\ell}}$ for which $J_{a} f_{k_{j_\ell}}$ converge a.\,e.\ to $J_{\rm a} f$. By the area formula from Corollary~\ref{T26} (i) and (ii) and Fatou's inequality, for any Borel set $E \subset \Omega$, we have
\begin{align} \label{1eq3}
\begin{split}
\liminf_{k \to \infty} |f_k(E)| &= \lim_{\ell \to \infty} |f_{k_{j_{\ell}}}(E)| \geq \lim_{\ell \to \infty} \int_{E} |J_{\rm a} f_{k_{j_{\ell}}}(x)| \, dx \\
&\geq \int_E |J_{\rm a} f(x)| \, dx \geq |f(E)|.
\end{split}
\end{align}

Now, Lemma~\ref{T9} implies that~\eqref{eq1} holds for compact sets $K \subset \Omega$. Inequality \eqref{1eq3} also says that for any Borel set $E$ compactly contained in $\Omega$, we have
$$
\infty > |f(\overline{E})| \geq \int_E |J_{\rm a} f(x)| \, dx \geq |f(E)|.
$$
This means that $J_{\rm a}f$ is locally integrable and, therefore, that $|f(E)|$ can be arbitrarily small provided that $|E|$ is sufficiently small.

Take any Borel set $E \Subset \Omega$ and fix $\eps >0$. We can find a compact set $K \subset E$ such that $E\setminus K$ is contained in a compact set $C$ with $|f(C)| < \eps$. We have
\begin{equation*}
\limsup_{k \to \infty} |f_k(E)| \leq \limsup_{k \to \infty} \left( |f_k(K)| + |f_k(C)| \right) \leq |f(K)| + |f(C)| \leq |f(E)| + \eps.
\end{equation*}

\noindent It follows from arbitrariness of $\eps$ and from~\eqref{1eq3} that~\eqref{eq1} holds for any Borel set $E \Subset \Omega$.
\end{proof}

\section{Proof of Theorem~\ref{1T2}} \label{s5}

In this section, we prove Theorem~\ref{1T2}. Firstly, we state a~folklore result, stating that for a~sequence of homeomorphisms, $W^{1,p}_{loc}$ convergence ($p > n-1$ for $n \geq 2$ and $p \geq 1$ for $n=2$) implies local uniform convergence. As we have not found a~reference to this fact, we include its proof after the short proof of Theorem~\ref{1T2}.
\begin{prop} \label{T16}
Suppose that $p > n-1$ for $n \geq 3$ and $p \geq 1$ for $n = 2$. Let $f_k: \Omega \to \R^n$ be homeomorphisms, $f: \Omega \to \R^n$ be continuous and let $f_k$ converge to $f$ in $W^{1,p}_{loc}(\Omega, \R^n)$. Then $f_k$ converge to $f$ locally uniformly on $\Omega$.
\end{prop}
\noindent The continuity assumption on $f$ cannot be dropped because of the standard example creating a \emph{cavity}. There is a sequence of homeomorphisms $f_k$ defined on the unit ball in $\R^n$ which converge in $W^{1,p}$ for $p < n$ to the discontinuous mapping $f$ given by $f(x) = x/|x|$ for $x \neq 0$ and $f(0) = 0$. One can take $f_k(x) = x (1+k)^{1/2}(1 +k|x|^2)^{-1/2}$, as in~\cite[Remark 2.11 (b)]{Distortion}.

\begin{proof}[Proof of Theorem~\ref{1T2}]
By Proposition~\ref{T16}, $f_k$ converge locally uniformly to $f$. Recall that every Sobolev mapping is a.\,e.\ approximately differentiable and the approximate and weak derivatives coincide. By Theorem~\ref{1T1}, \eqref{eq1} holds for any Borel set $E \Subset \Omega$. As observed in Remark~\ref{T27}, area formula implies that the Jacobian of a homeomorphism is always locally integrable, hence the $L^1_{loc}(\Omega)$-convergence of Jacobians follows from Theorem~\ref{T1}.
\end{proof}

We start the proof of Proposition~\ref{T16} by a~technical lemma similar to \cite[Lemma 26]{GoldsteinHajlasz19}.

\begin{lemma} \label{T15}
Let $p \geq 1$, $K$ be a compact subset of $\Omega$ and $g_k: \Omega \to \R^n$, $g_k \in W^{1,p}_{loc}(\Omega, \R^n)$, be a sequence of mappings converging to zero in $W^{1,p}_{loc}(\Omega, \R^n)$. Then for any $\eta >0$, it is possible to find a finite family of closed cubes $\{Q_i\}_{i=1}^N$ and a subsequence $g_{k_j}$ such that
\begin{enumerate}[(i)]
    \item $K \subset \bigcup_{i=1}^N Q_i \subset \Omega$ and $\diam Q_i < \eta$;
    \item $g_{k_j} \in W^{1,p}(\partial Q_i, \R^n)$ for $i = 1, \ldots, N$;
    \item $\lVert g_{k_j} \rVert_{1, p, \partial Q_i} \xrightarrow{j \to \infty} 0$ for every $i = 1, \ldots, N$.
\end{enumerate}
\end{lemma}
Here, the mappings $g_k$ are defined everywhere, i.\,e., they are fixed representatives. We will use this lemma for $g_k = f_k -f$, where $f_k, f$ are homeomorphisms from Theorem~\ref{1T2} and we do not want to modify $f_k, f$ even on a set of measure zero.

\begin{proof}
We find $\eps >0$ such that $K_{4\eps} = \{x: \, \dist(x, K) < 4\eps \}$ is compactly contained in $\Omega$. Without loss of generality, we can assume that $\eta < \eps$. Consider the cubes centered at $\eta/(2\sqrt{n}) \mathbb{Z}^n$ with sides parallel to the coordinate axes. We take those which have nonempty intersection with $K_\eta$ and call $G$ the grid consisting of the boundaries of these cubes.

Note that the diameter of each such cube equals $\eta/2 < \eta$. Even if $G$ is translated by a vector $v$ with $|v| < \eta$, the cubes, whose boundaries constitute $G + v$, cover $K$ and are contained in $K_{4\eps}$ and hence in $\Omega$. The rest of the proof is devoted to finding an appropriate translation vector $v_o$.

Suppose that $F: K_{4\eps} \to [0, \infty[$, $B_{\eta}(0) := \{x\in \R^n: |x| < \eta\}$. Following \cite[p.\ 135]{White}, we have 
\begin{align} \label{eq52}
\begin{split}
    \int_{B_\eta(0)} \int_G F(x + v) \, dx \, dv &= \int_G \int_{B_\eta(0)} F(x + v) \, dv \, dx = \int_G \int_{G_\eta} F(z) \, dz \, dx\\
    &\leq \mathcal{H}^{n-1}(G) \int_{K_{4\eps}} F(z) \, dz.
\end{split}
\end{align}
In the first equality, we use Fubini's theorem and in the second---simple change of variables.

For each $k$, we can find a sequence of smooth mappings $h_k^m \in C^\infty(K_{4 \eps}, \R^n) \cap W^{1,p}(K_{4\eps}, \R^n)$ convergent to $g_k$ in $W^{1,p}(K_{4\eps}, \R^n)$. Using \eqref{eq52} for
$$
F(x) = |g_k(x) - h_k^m(x)|^p + |Dg_k(x) - Dh_k^m(x)|^p,
$$
we see that when $m \to \infty$, the function $x \mapsto \int_G F(x+v) \, dx$ converges in $L^1(B_\eta(0))$ to zero. Therefore, there is a subsequence (which we do not relabel) which converges to zero for a.\,e.\ $v \in B_\eta(0)$. We denote with $V_k$ the set of $v \in B_\eta(0)$, for which this pointwise convergence takes place. For all $k$, $|V_k| = |B_\eta(0)|$.

For any $v \in V_k$, we have found a sequence of smooth mappings convergent to $g_k$ in $W^{1,p}(G+v, \R^n)$. Completeness of Sobolev spaces implies that $g_k \in W^{1,p}(G+v, \R^n)$. Moreover, a standard argument shows that
$$
D \left( g_k |_{G + v} \right) = Dg_k |_{G + v},
$$
see the proof of \cite[Section 5.2, Theorem 2]{EvansPDE} for details.

Now, using \eqref{eq52} for $F(x) = |g_k(x)|^p + |Dg_k(x)|^p$ and the same argument as before, we find a subsequence (which we relabel this time) $k_j$ such that, when $j \to \infty$, the function
\begin{equation} \label{eq53}
x \mapsto \int_G |g_{k_j}(x+v)|^p + |Dg_{k_j}(x+v)|^p \, dx
\end{equation}
converges to zero for a.\,e.\ $v \in B_\eta(0)$. Hence, we can find $v_o \in \bigcap_{k=1}^\infty V_k$ for which also \eqref{eq53} converges to zero. Indeed, since each $V_k$ was of full measure, so is its countable intersection.

The set $G + v_o$ consists of boundaries of the required cubes. It follows from the reasoning presented at the beginning that (i) holds. Properties (ii) and (iii) follow directly from the construction.
\end{proof}

\begin{proof}[Proof of Proposition~\ref{T16}]
Given a~compact set $K \subset \Omega$, we shall show that from every subsequence of $f_k$, which we do not relabel, we can choose a~further subsequence, which converges to $f$ uniformly on $K$. We fix $\eps >0$ and choose $\eta \in (0, \eps)$ from uniform continuity of $f$ on $K$ so that if $|x - y| < \eta$, $|f(x) - f(y)| < \eps$. For this $\eta$ and the sequence $f_k - f$, by Lemma~\ref{T15}, we find cubes $Q_i$, $i = 1, \ldots, N$, and a~subsequence $f_{k_j} - f$ satisfying
\begin{enumerate}[(i)]
    \item $K \subset \bigcup_{i=1}^N Q_i \subset \Omega$, $\diam Q_i < \eta$;
    \item $f_{k_j} - f \in W^{1,p}(\partial Q_i, \R^n)$ for $i = 1, \ldots, N$;
    \item $\lim_{j \to \infty} \lVert f_{k_{j}} - f \rVert_{1,p,\partial Q_i} = 0$ for $i = 1, \ldots, N$.
\end{enumerate}
For each $i$, we can choose $j_i$ so that for $j \geq j_i$,
\begin{equation} \label{eq19}
    \lVert f_{k_j} - f \rVert_{1,p, \partial Q_i} \leq \eps.
\end{equation}
Moreover, we can choose yet another subsequence out of $f_{k_j}$, which we do not relabel, which converges $\Hs^{n-1}$-a.\,e.\ on $\partial Q_i$. For each $i=1, \ldots, N$, we find $x_i \in \partial Q_i$ so that (possibly by enlarging $j_i$),
\begin{equation} \label{eq17}
| f_{k_j}(x_i) - f(x_i)| \leq \eps.
\end{equation}
We can assume that $x_i$ is a Lebesgue point for $f_{k_j} -f$, which will be useful in applying Morrey's inequality below. Setting $j_o:=\max_{1 \leq i \leq N} j_i$ means that \eqref{eq19} and \eqref{eq17} hold for $j \geq j_o$ for all $i = 1, \ldots, N$.

We shall now consider the case $n \geq 3$; we will use Morrey's inequality on $\partial Q_i$, see \cite[Section 4.5.3, Theorem 3]{EvansGariepy}. We can do it because $\partial Q_i$ is bi-Lipschitz equivalent to a ball in $\R^{n-1}$ and $p > n-1$. Recall that Morrey's inequality holds for any two Lebesgue points of a given function\footnote{See the proof of \cite[Thorem 3.23]{Kinnunen} for a clear explanation of this fact.} so that for a.\,e. $x \in \partial Q_i$ and the chosen Lebesgue point $x_i$, we have
\begin{align} \label{eq50}
\begin{split}
|f_{k_j}(x) - f(x) - f_{k_j}(x_i) + f(x_i)| &\leq C_i \diam(\partial Q_i)^{1 - (n-1)/p} \lVert f_{k_j} - f \rVert_{1, p, \partial Q_i} \\
& \leq C_i \eps^{1 - (n-1)/p} \, \eps = C_i \eps^{2 - (n-1)/p}.
\end{split}
\end{align}
The constant $C_i$ depends on $n,\, p$ and the Lipschitz constant in the $\partial Q_i$-to-ball equivalence. We used here also~\eqref{eq19} and (ii).

Collecting \eqref{eq50} and \eqref{eq17} yields for a.\,e.\ $x \in \partial Q_i$ and for all $j \geq j_o$,
\begin{equation*}
|f_{k_j}(x) - f(x)| \leq C_i \eps^{2 -(n-1)/p} + \eps \leq \tau/2, \, \text{ where } \tau:= 2(\max_{1 \leq i \leq N} C_i \eps^{2 - (n-1)/p} + \eps).
\end{equation*}
In fact, continuity of $f_{k_j}$ and triangle inequality imply that even
\begin{equation} \label{eq20a}
    \sup_{x \in \partial Q_i} |f_{k_j}(x) - f(x)| \leq \tau.
\end{equation}
Inequality~\eqref{eq20a} implies that $f_{k_j}(\partial Q_i) \subset f(\partial Q_i)_{\tau} \subset f(Q_i)_{\tau}$ for all $j \geq j_o$ and for all $i = 1, \ldots, N$. By Jordan separation theorem \cite[Theorem 3.29]{FonsecaGangbo}, $f_{k_j}(\partial Q_i)$ is homeomorphic to $\mathbb{S}^{n-1}$ and divides $\R^{n}$ into two connected components, one bounded and one unbounded. Since $f_{k_j}(Q_i)$ is compact, it must be contained in the closure of the bounded component and thus $f_{k_j}(Q_i) \subset f(Q_i)_{\tau} \text{ for all } i = 1, \ldots, N \text{ and } j \geq j_o.$

Consequently, we get for all $i = 1, \ldots, N$ and $j \geq j_o$,
$$
\sup_{x \in Q_i} |f_{k_j}(x) - f(x)| \leq \diam f(Q_i)_{\tau} \leq \diam f(Q_i) + \tau \leq \eps + \tau.
$$
Clearly, $\eps$ is arbitrary and $\tau$ converges to zero with $\eps$ going to zero and $K \subset \bigcup_{i=1}^N Q_i$. This finishes the proof of local uniform convergence in the case $n \geq 3$.

In the case $n=2$ and $p \geq 1$, we use the fundamental theorem of calculus instead of Morrey's inequality. Squares $Q_i$ have been chosen to guarantee that $f_{k_j}$ is absolutely continuous on $\partial Q_i$ for every $j$. Indeed, $f_{k_j} \in W^{1,1}(\partial Q_i)$ and since $\partial Q_i$ is essentially a segment, it follows that $f_{k_j}$ is absolutely continuous. Therefore, for any $x \in \partial Q_i$,
$$
|f_{k_j}(x) - f(x)| \leq \lVert f_{k_j} - f \rVert_{1,1,\partial Q_i} + |f_{k_j}(x_i) - f(x_i)| \leq \tilde{C}_i \eps.
$$
Above, we used also \eqref{eq19} and \eqref{eq17} and H\"{o}lder's inequality. The rest of the argument now follows like in the case $n \geq 3$. This finishes the proof.
\end{proof}

\section{Proof of Theorem~\ref{T6}} \label{s6}
To keep the paper self-contained, we present below the basics of degree theory in Euclidean spaces, following the book \cite{FonsecaGangbo}.
\begin{defn}
Let $f \in C^1(\overline{\Omega}, \R^n)$, i.\,e., we assume that $f$ admits an extension $\tilde{f}$ to an open set containing $\overline{\Omega}$ on which $\tilde{f}$ is $C^1$. Set $Z_f := \{x \in \overline{\Omega}: \, Jf(x) = 0 \}$ and suppose that $p \in \R^n \setminus (f(Z_f) \cup f(\partial \Omega))$. Then,
\begin{equation} \label{eq36}
\degg (f, \Omega, p) := \sum_{x \in f^{-1} (p)} \mathrm{sgn} (Jf(x)).
\end{equation}
\end{defn}
\noindent Degree for continuous mappings can be defined through smooth approximations, we refer to \cite[Chapter 1]{FonsecaGangbo} for details. 

In the sequel, we will need the following facts about degree, which we gather below for the convenience of the reader. Properties (i) and (ii) are Theorem 2.1 and Theorem 2.3 (1) from \cite{FonsecaGangbo}, respectively.
\begin{prop} \label{T8}
Suppose that $f, g: \overline{\Omega} \to \R^n$ are continuous.
\begin{enumerate}[(i)]
    \item If $p \notin f(\partial \Omega)$ and $\deg (f, \Omega, p) \neq 0$, then there is $x \in \Omega$ with $f(x) = p$.
    \item If $\sup_{x \in \Omega} |g(x) - f(x) | <  \dist(y, f(\partial \Omega))/\sqrt{n}$, then\footnote{The factor $1/\sqrt{n}$ appears here and does not appear in \cite{FonsecaGangbo}, because in \cite{FonsecaGangbo}, $|x|$ for $x \in \R^n$ denotes the maximum norm and this norm is also used to define the distance between a point and a set (denoted there with $\rho$). Here, we use the $2$-norm in both instances.}
    $$\deg(g, \Omega, y) = \deg(f, \Omega, y).$$
\end{enumerate}
\end{prop}

\begin{proof}[Proof of Theorem \ref{T6}]
Fix $\eps > 0$ and observe that for any subset $A \Subset \Omega$, we can find $\eta$ such that if $E$ is a measurable subset of $A$ with $|E| < \eta$, then
\begin{equation} \label{eq23}
|f_k(E)| \leq \int_{E} |J_{\rm a} f_k(x)| \, dx \leq \eps.
\end{equation}
This follows from Corollary~\ref{T26} (i) and local uniform integrability of $\{J_{\rm a}f_k\}_k$. Consequently, given any measurable set $E \Subset \Omega$, we can find a compact subset $K_\eps \subset E$ with $|E \setminus K_{\eps}|$ sufficiently small so that
\begin{equation*}
|f_k(K_{\eps})| \leq |f_k(E)| \leq |f_k(K_\eps)| + \eps.
\end{equation*}
It follows from the arbitrariness of $\eps$ that if \eqref{eq1} holds for all compact subsets of $\Omega$, then it also holds for all measurable sets compactly contained in $\Omega$. Also, in view of Lemma~\ref{T9}, it suffices to prove that for any compact set $K \subset \Omega$,
\begin{equation} \label{eq24}
    \liminf_{k \to \infty} |f_k(K)| \geq |f(K)|.
\end{equation}

Fix $\eps >0$ and a~compact set $K \subset \Omega$. We find an open and connected set $U \Subset \Omega$ such that
$$
K \subset U, \quad |\partial U| = 0, \quad |U \setminus K| < \eta
$$
for a~sufficiently small $\eta >0$, which by~\eqref{eq23} guarantees that
\begin{equation} \label{eq54}
|f_k(U)| \leq |f_k(K)| + \eps \text{ for all } k \in \mathbb{N}.
\end{equation}
One can construct such $U$ by covering $K$ with a~finite number of sufficiently small balls and joining the connected components of that covering with thin tubes. Since $f$ satisfies the Lusin condition (N), $f(\partial U)$ is a~set of measure zero. 

Consequently, the measure of the set $\{x:\, \dist(x, f(\partial U)) < \delta\}$ converges to zero as $\delta \to 0$. We can find a $\delta_o>0$ for which the set $G := \{x: \, \dist(x, f(\partial U)) < \delta_o\}$ has measure $|G| < \eps$ and $G \Subset \Omega$. Moreover, for any $y \in f(U) \setminus G$, we have $\dist(y, f(\partial U)) \geq \delta_o$.

Since $f_k$ is uniformly convergent to $f$ on $U$ (as $U \Subset \Omega)$, we can find $k_o$ such that for all $k \geq k_o$, $\sup_{x \in U} |f_k(x) - f(x)| < \delta_o/\sqrt{n}$. By the choice of $\delta_o$ and Proposition~\ref{T8} (ii), this implies that 
\begin{equation} \label{eq55}
\deg(f_k, U, y) = \deg(f, U, y) \quad \text{for all } k \geq k_o, \, y \in f(U) \setminus G.
\end{equation}
Note that Proposition~\ref{T8} (ii) says, in particular, that $\deg(f_k, U, y)$ is well defined for such $y$ and $k$, i.\,e., that $y \not \in f_k(\partial U)$.

By \eqref{eq55} and assumption~\eqref{eq4} in the statement of the theorem, for almost any $y \in f(U) \setminus G$ and $k \geq k_o$, we have $\deg(f_k, U, y) \neq 0$. This implies that for a.\,e.\ such $y$, there exists $x \in U$ s.\,t.\ $f_k(x) = y$, see Proposition~\ref{T8}~(i). Consequently, for sufficiently large $k$, almost every point of $f(U) \setminus G$ is contained in $f_k(U)$, which allows to conclude that
$$
    |f(U) \setminus G| \leq |f_k(U)|.
$$
It then follows from the choice of $U$ and $G$ and from \eqref{eq54} that for $k \geq k_o$
$$
|f(K)| - \eps \leq |f(U)| - |G| \leq |f(U) \setminus G| \leq |f_k(U)| \leq |f_k(K)| + \eps.
$$
Since $\eps$ is arbitrary, $\liminf_{k \to \infty} |f_k(K)| \geq |f(K)|$, which finishes the proof of~\eqref{eq24} and hence the proof.
\end{proof}

It is well known that condition \eqref{eq4} is satisfied for any homeomorphism, see \cite[Theorem 3.35]{FonsecaGangbo}. Below, we shall give a~weaker condition which guarantees \eqref{eq4} and to this end, we need to use the notion of regular approximate derivative.
\begin{defn}
Let $f: \Omega \to \R^n$ be measurable. We say that $f$ is regularly approximately differentiable at a~point $x \in \Omega$ if $f$ is approximately differentiable at point $x$ with derivative $D_{\rm a} f(x)$ and if there exists a~sequence $r_j \to 0$ s.\,t.\ 
$$
\lim_{j \to \infty} \sup_{| h | = 1} \frac{1}{r_j} | f(x + r_j h) - f(x) - D_{\rm a}f(x) r_j h | = 0.
$$
We use the same notation for approximate derivative which is not necessarily regular.
\end{defn}

Reshetnyak in~\cite{Resh67} in 1967 proved that $f \in W^{1,p}(\Omega, \R^n)$ for $p > n-1$ has regular approximate differential a.\,e in $\Omega$, see \cite{GoffmanZiemer} for an alternative proof. Regularity of approximate derivative is enough for the analogue of the classical degree formula \eqref{eq36} to hold.

\begin{prop} \cite[p.\ 284]{GoldResh} \label{T22}
Let $f: \Omega \to \R^n$ be a~continuous mapping, which satisfies the Lusin condition (N), is regularly approximately differentiable a.\,e.\ and its approximate Jacobian $J_{\rm a} f$ is locally integrable. For any bounded domain $ U \Subset \Omega$ with $|\partial U| = 0$, for almost every $y \in \R^n$, $J_{\rm a} f(x)$ is well defined for $x \in f^{-1}(y)$ and the formula
$$
\deg (f, U, y) = \sum_{x \in f^{-1}(x) \cap U} \mathrm{sgn} \, ( J_{\rm a} f(x) )
$$
holds.
\end{prop}
The assumption on regularity of the approximate derivative is necessary, as shown by the example described in \cite{GoldsteinHajlasz17} of an a.\,e.\ approximately differentiable homeomorphism $\Phi$ of the unit cube such that $\degg(\Phi, (0,1)^n, y) = 1$ for any $y \in (0,1)^n$ and whose approximate Jacobian equals $-1$ a.\,e.

Eventually, we get the following
\begin{cor}
If $f: \Omega \to \R^n$ is as in Proposition \ref{T22} and 
$$
J_{\rm a} f > 0 \text{ or } J_{\rm a} f < 0 \text{ a.\,e.\ on } \Omega,
$$
then condition~\eqref{eq4} holds.    
\end{cor}

\bibliography{bib}
\bibliographystyle{abbrv}
\end{document}